\documentclass{article}
\usepackage[a4paper]{geometry}
\usepackage[all,cmtip]{xy}
\usepackage[utf8]{inputenc}
\usepackage[initials]{amsrefs}
\usepackage[labelsep=period]{caption}
\usepackage{subcaption}
\usepackage{amsfonts,authblk,bm,doi,enumitem,fancyhdr,float,graphicx,amssymb,amsthm,mathrsfs,mleftright,relsize,sectsty,subcaption,tikz,titlesec}
\usepackage[fleqn]{amsmath}
\usepackage[svgnames]{xcolor}
\usetikzlibrary{calc,decorations.pathreplacing}

\newif\ifshowchanges 
 \showchangesfalse 

\newcommand{\rev}[1]{%
	\ifshowchanges \textcolor{blue}{#1}%
	\else #1%
\fi }

\newtheorem{theorem}{Theorem}[section]
\newtheorem{lemma}[theorem]{Lemma}
\theoremstyle{definition}
\theoremstyle{plain}

\newtheorem{innerthm}{Theorem}
\newenvironment{maintheorem}[1]
  {\renewcommand\theinnerthm{#1}\innerthm}
  {\endinnerthm}
  
\newtheorem{innercor}{Corollary}
\newenvironment{maincorollary}[1]
  {\renewcommand\theinnercor{#1}\innercor}
  {\endinnercor}
  
\numberwithin{equation}{section}
\numberwithin{figure}{section}

\renewcommand{\geq}{\geqslant}
\renewcommand{\leq}{\leqslant}

\setlist[enumerate]{leftmargin=20pt,itemsep=0pt,topsep=0pt}
\setlist[enumerate,1]{label=\textnormal{(\roman*)}}
\setlist[itemize]{leftmargin=20pt,itemsep=0pt,topsep=0pt}

\makeatletter
\renewcommand\section{\@startsection {section}{1}{\z@}%
                                   {-3.5ex \@plus -1ex \@minus -.2ex}%
                                   {1.3ex \@plus.2ex}%
                                   {\normalfont\large\scshape}}
\makeatother

\tikzset{
  horocycle/.style={draw, fill=LightSteelBlue},
  geodesic/.style={thin},
  finite edge/.style={very thick},
}

\newcommand{\IdealPointAngle}[2]{%
  \pgfmathsetmacro{\tmpx}{cos(#2)}%
  \pgfmathsetmacro{\tmpy}{sin(#2)}%
  \coordinate (#1) at (\tmpx,\tmpy);
  \expandafter\xdef\csname ptx@#1\endcsname{\tmpx}%
  \expandafter\xdef\csname pty@#1\endcsname{\tmpy}%
}

\newcommand{\IdealPointXY}[3]{%
  \pgfmathsetmacro{\tmplength}{veclen(#2,#3)}%
  \pgfmathsetmacro{\tmpx}{#2/\tmplength}%
  \pgfmathsetmacro{\tmpy}{#3/\tmplength}%
  \coordinate (#1) at (\tmpx,\tmpy);
  \expandafter\xdef\csname ptx@#1\endcsname{\tmpx}%
  \expandafter\xdef\csname pty@#1\endcsname{\tmpy}%
}

\newcommand{\SetHorocycleRadius}[2]{%
  \expandafter\xdef\csname hr@#1\endcsname{#2}%
}

\newcommand{\DrawHorocycle}[1]{%
  \edef\rr{\csname hr@#1\endcsname}%
  \coordinate (#1-h) at ($(0,0)!{1-\rr}!(#1)$);
  \draw[horocycle] (#1-h) circle[radius=\rr];
}

\newcommand{\DrawEmptyHorocycle}[1]{%
  \edef\rr{\csname hr@#1\endcsname}%
  \coordinate (#1-h) at ($(0,0)!{1-\rr}!(#1)$);
  \draw[] (#1-h) circle[radius=\rr];
}

\newcommand{\IdealLabel}[3][]{%
  \node[#1] at ($(0,0)!0.98!(#2)$) {#3};
}

\newcommand{\Geodesic}[4][]{%
  \edef\ax{\csname ptx@#2\endcsname}%
  \edef\ay{\csname pty@#2\endcsname}%
  \edef\bx{\csname ptx@#3\endcsname}%
  \edef\by{\csname pty@#3\endcsname}%
  \edef\ra{\csname hr@#2\endcsname}%
  \edef\rb{\csname hr@#3\endcsname}%

  \pgfmathsetmacro{\detAB}{\ax*\by-\ay*\bx}
  \pgfmathparse{abs(\detAB) < 0.0001 ? 1 : 0}%
  \ifnum\pgfmathresult=1
    \pgfmathsetmacro{\sx}{(1-2*\ra)*\ax}
    \pgfmathsetmacro{\sy}{(1-2*\ra)*\ay}
    \pgfmathsetmacro{\tx}{(1-2*\rb)*\bx}
    \pgfmathsetmacro{\ty}{(1-2*\rb)*\by}

    \draw[geodesic] (#2) -- (#3);
    \draw[finite edge] ({\sx},{\sy}) -- ({\tx},{\ty});
    \path ({0.5*(\sx+\tx)},{0.5*(\sy+\ty)}) node[#1] {#4};
  \else
    \pgfmathsetmacro{\cx}{(\by-\ay)/\detAB}
    \pgfmathsetmacro{\cy}{(\ax-\bx)/\detAB}
    \pgfmathsetmacro{\R}{veclen(\ax-\cx,\ay-\cy)}

    \pgfmathsetmacro{\alpha}{atan2(\ay-\cy,\ax-\cx)}
    \pgfmathsetmacro{\beta}{atan2(\by-\cy,\bx-\cx)}
    \pgfmathsetmacro{\kappa}{mod(\beta-\alpha+540,360)-180}

    \pgfmathsetmacro{\hx}{(1-\ra)*\ax}
    \pgfmathsetmacro{\hy}{(1-\ra)*\ay}
    \pgfmathsetmacro{\mx}{\hx-\cx}
    \pgfmathsetmacro{\my}{\hy-\cy}
    \pgfmathsetmacro{\nx}{-\my}
    \pgfmathsetmacro{\ny}{\mx}
    \pgfmathsetmacro{\nn}{\nx*\nx+\ny*\ny}
    \pgfmathsetmacro{\tau}{-2*(\nx*(\ax-\cx)+\ny*(\ay-\cy))/\nn}
    \pgfmathsetmacro{\sx}{\ax+\tau*\nx}
    \pgfmathsetmacro{\sy}{\ay+\tau*\ny}

    \pgfmathsetmacro{\kx}{(1-\rb)*\bx}
    \pgfmathsetmacro{\ky}{(1-\rb)*\by}
    \pgfmathsetmacro{\mxb}{\kx-\cx}
    \pgfmathsetmacro{\myb}{\ky-\cy}
    \pgfmathsetmacro{\nxb}{-\myb}
    \pgfmathsetmacro{\nyb}{\mxb}
    \pgfmathsetmacro{\nnb}{\nxb*\nxb+\nyb*\nyb}
    \pgfmathsetmacro{\tauB}{-2*(\nxb*(\bx-\cx)+\nyb*(\by-\cy))/\nnb}
    \pgfmathsetmacro{\tx}{\bx+\tauB*\nxb}
    \pgfmathsetmacro{\ty}{\by+\tauB*\nyb}

    \pgfmathsetmacro{\alphas}{atan2(\sy-\cy,\sx-\cx)}
    \pgfmathsetmacro{\betas}{atan2(\ty-\cy,\tx-\cx)}
    \pgfmathsetmacro{\kappas}{mod(\betas-\alphas+540,360)-180}

    \draw[geodesic] ({\cx},{\cy}) ++({\alpha}:{\R})
      arc[start angle=\alpha, delta angle=\kappa, radius=\R];
    \draw[finite edge] ({\cx},{\cy}) ++({\alphas}:{\R})
      arc[start angle=\alphas, delta angle=\kappas, radius=\R];
    \path ({\cx},{\cy}) ++({\alphas+0.5*\kappas}:{\R}) node[#1] {#4};
  \fi
}

\tikzset{
  disk boundary/.style={thick},
  ptolemy edge/.style={thick},
  ptolemy diagonal/.style={dashed},
  ptolemy point/.style={fill=black},
  case circle/.style={draw},
  case tangent/.style={thin},
  case line/.style={thin},
  case label/.style={fill=white, inner sep=0.5pt},
  common point/.style={fill=black},
}

\newcommand{\DrawDiskBoundary}[1][]{%
  \draw[disk boundary,#1] (0,0) circle[radius=1];%
}

\newcommand{\PtolemyVertex}[4][]{%
  \IdealPointAngle{#2}{#3}%
  \filldraw[ptolemy point] (#2) circle[radius=1pt];%
  \node[#1] at ($(0,0)!1.02!(#2)$) {#4};
}

\newcommand{\SegmentLabel}[5][]{%
  \path ($(#2)!#3!(#4)$) node[#1] {#5};%
}

\newcommand{\LabelledHorocycle}[3][]{%
  \DrawEmptyHorocycle{#2}%
  \node[#1] at (#2-h) {#3};%
}

\newcommand{\HorocycleTangent}[5][]{%
  \edef\ax{\csname ptx@#2\endcsname}%
  \edef\ay{\csname pty@#2\endcsname}%
  \edef\bx{\csname ptx@#3\endcsname}%
  \edef\by{\csname pty@#3\endcsname}%
  \edef\ra{\csname hr@#2\endcsname}%
  \edef\rb{\csname hr@#3\endcsname}%

  \pgfmathsetmacro{\xA}{(1-\ra)*\ax}%
  \pgfmathsetmacro{\yA}{(1-\ra)*\ay}%
  \pgfmathsetmacro{\xB}{(1-\rb)*\bx}%
  \pgfmathsetmacro{\yB}{(1-\rb)*\by}%
  \pgfmathsetmacro{\dx}{\xB-\xA}%
  \pgfmathsetmacro{\dy}{\yB-\yA}%
  \pgfmathsetmacro{\dr}{\ra-\rb}%
  \pgfmathsetmacro{\Lsq}{\dx*\dx+\dy*\dy}%
  \pgfmathsetmacro{\h}{sqrt(max(0,\Lsq-\dr*\dr))}%
  \pgfmathsetmacro{\nx}{(\dr*\dx+(#4)*\h*(-\dy))/\Lsq}%
  \pgfmathsetmacro{\ny}{(\dr*\dy+(#4)*\h*( \dx))/\Lsq}%
  \pgfmathsetmacro{\sx}{\xA+\ra*\nx}%
  \pgfmathsetmacro{\sy}{\yA+\ra*\ny}%
  \pgfmathsetmacro{\tx}{\xB+\rb*\nx}%
  \pgfmathsetmacro{\ty}{\yB+\rb*\ny}%

  \draw[case tangent] ({\sx},{\sy}) -- ({\tx},{\ty})
    node[midway,#1] {#5};%
}

\newcommand{\CircleThroughPoint}[6][]{%
  \coordinate (#3-center) at ($(#2)+(#4:#5)$);%
  \coordinate (#3-label) at ($(#2)+1.6*(#4:#5)$);%
  \draw[case circle] (#3-center) circle[radius=#5];%
  \node[#1] at (#3-label) {#6};%
}

\newcommand{\LineThroughPointRange}[5][]{%
  \draw[case line,#1] ($(#2)+({#3+180}:#4)$) -- ($(#2)+(#3:#5)$);%
}

\newcommand{\NestedEqualAngleCircle}[7][]{%
  \pgfmathsetmacro{\CaseThreeRadius}{abs(\CaseThreeCircleRatio*(#5))}%
  \coordinate (#3-center) at ($(#2)+(#4:#5)$);%
  \draw[case circle] (#3-center) circle[radius=\CaseThreeRadius];%
  \path (#3-center) ++(#6:{1*\CaseThreeRadius+0.1})
    node[case label,#1] {#7};%
}

\title{ \vspace{-6ex}\bf \large A unified approach to Penner, Ptolemy, and Casey's theorems in several dimensions}

\makeatletter
\renewcommand*{\@fnsymbol}[1]{\hspace*{-10pt}}
\makeatother

\author{Isabella Lewis and Ian Short}

\date{\vspace{-5ex}}

\begin{document}

\maketitle

\begin{abstract}
We prove Penner's theorem on horocycles and theorems of Ptolemy and Casey, all with full converses, in hyperbolic space of several dimensions. Recently Waddle observed that the equations underpinning these three theorems are related, and it is this viewpoint that we advance, using the Lorentzian model of hyperbolic space. We show that all three theorems can be derived from a common Gram-matrix calculation applied to  lightlike, timelike, and spacelike vectors. Remarkably, our approach gives a version of Casey's theorem in the plane with a full converse, involving three geometric alternatives, which to our knowledge has not previously been recorded.
\end{abstract}

\section{Introduction}

The objective of this partly expository paper is to prove versions of Penner's theorem on lambda lengths, Ptolemy's theorem, and Casey's theorem -- and the converse of each theorem -- in hyperbolic space of several dimensions. Many known variants of Ptolemy and Casey's theorems for Euclidean, spherical, and hyperbolic geometry are special cases.

Our starting point is Penner's theorem on lambda lengths between horocycles in the hyperbolic plane \cites{Pe1987,Pe2006}. This theorem was generalised to three dimensions by Felikson et al.\ in \cite{FeKaSeTu2023}, and one of the methods of those authors (which they attribute to Izmestiev \cite{FeKaSeTu2023}*{Remark~4.19}) gives a version of Penner's theorem in all dimensions. Alongside this, Mathews offered a version of Penner's theorem for \emph{complex} lambda lengths of spin-decorated horospheres in \cite{Ma2025}.  Separately, there are various works on Ptolemy's and Casey's theorems in higher dimensions for Euclidean, spherical, and hyperbolic geometry. These were surveyed and advanced by Maehara and Martini in \cite{MaMa2019}, who observe that Casey's theorem\rev{ -- proven by Casey in 1866 --  was in fact known to Shiraishi in the 1820s}. Waddle \cite{Wa2025} demonstrated that the equations underlying the theorems of Penner, Ptolemy, and Casey (and a result of Pl\"ucker) can be obtained from one another, and it is this perspective that we develop, using the Lorentzian model of hyperbolic space. We will see that the three theorems are variants on a common result on Gram matrices in Lorentzian geometry, which we apply for collections of lightlike (Penner), timelike (Ptolemy), and spacelike (Casey) vectors. 

Let $H^n$ denote $n$-dimensional hyperbolic space, for $n\geq 2$, with distance function $\varrho$. \rev{We refer the reader to \cites{Pe2006,Ra2019} for the various realisations of constant curvature spaces described in this introduction.} Our first result is a reformulation of the generalisation of Penner's theorem due to Felikson et al. We recall (from \cite{Pe2006}) that the \emph{$\lambda$-length}  $\lambda(\Sigma,\Sigma')$ between two horospheres $\Sigma$ and $\Sigma'$ in $H^n$ with distinct centres is $e^{\delta/2}$, where $\delta$ is the signed hyperbolic distance between $\Sigma$ and $\Sigma'$. We define $\lambda(\Sigma,\Sigma')=0$ if the horospheres share the same centre.

\begin{maintheorem}{A}[Penner's theorem]\label{theoremA}
The centres of the horospheres $\Sigma_1,\Sigma_2,\dots,\Sigma_{n+1}$ in $H^{n}$ lie on the ideal boundary of a common hyperplane if and only if $\det A=0$, where $A_{ij}=\lambda(\Sigma_i,\Sigma_j)^2$.
\end{maintheorem}

To recover Penner's original theorem from Theorem~\ref{theoremA} we consider the case $n=3$ (\emph{not} $n=2$, with this formulation). Let $S_1$, $S_2$, $S_3$, and $S_4$ be four horocycles in $H^2$. We can embed $H^2$ as a hyperplane $\Pi$ in $H^3$. Let $\Sigma_i$ be the horosphere in $H^3$ that intersects $\Pi$ orthogonally in $S_i$, for $i=1,2,3,4$. Then the centres of $\Sigma_i$ lie on the ideal boundary of $\Pi$, so $\det A=0$, by Theorem~\ref{theoremA}. Writing $\lambda_{ij}=\lambda(\Sigma_i,\Sigma_j)$, we have that $\det A=-\Lambda (\lambda_{14} \lambda_{23} + \lambda_{13} \lambda_{24} + \lambda_{12} \lambda_{34})$, where
\[
\Lambda =(\lambda_{14} \lambda_{23} + \lambda_{13} \lambda_{24} - \lambda_{12} \lambda_{34})(\lambda_{14} \lambda_{23} - \lambda_{13} \lambda_{24} + \lambda_{12} \lambda_{34})(-\lambda_{14} \lambda_{23} + \lambda_{13}\lambda_{24} + \lambda_{12} \lambda_{34}).
\] 
This gives us Penner's theorem in $H^2$, illustrated in Figure~\ref{figure1}, which says that for any four horocycles $S_1$, $S_2$, $S_3$, and $S_4$ in $H^2$ we have one of $\lambda_{12} \lambda_{34}=\lambda_{14} \lambda_{23} + \lambda_{13} \lambda_{24}$, $\lambda_{13} \lambda_{24}=\lambda_{14} \lambda_{23} + \lambda_{12} \lambda_{34}$, or $\lambda_{14} \lambda_{23}=\lambda_{13}\lambda_{24}+\lambda_{12} \lambda_{34}$. By following this same reasoning with $n=4$ rather than $n=3$ we can obtain \cite{FeKaSeTu2023}*{Theorem~4.18}.

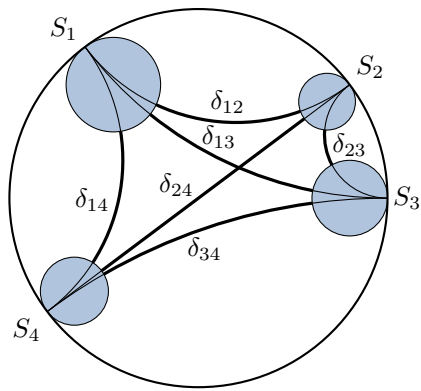
\begin{figure}[ht]
\centering
\begin{tikzpicture}[scale=2.5]


  \IdealPointXY{A}{-0.6}{ 0.8}
  \IdealPointXY{B}{ 0.8}{ 0.6}
  \IdealPointAngle{C}{0}
  \IdealPointXY{D}{-0.8}{-0.6}


  \SetHorocycleRadius{A}{0.25}
  \SetHorocycleRadius{B}{0.15}
  \SetHorocycleRadius{C}{0.20}
  \SetHorocycleRadius{D}{0.18}


  \begin{scope}
    \clip (0,0) circle[radius=1];

    \DrawHorocycle{A}
    \DrawHorocycle{B}
    \DrawHorocycle{C}
    \DrawHorocycle{D}

    \Geodesic[above]{A}{B}{$\delta_{12}$}
    \Geodesic[right,xshift=-1pt,yshift=1pt]{B}{C}{${\delta}_{23}$}
    \Geodesic[below]{C}{D}{$\delta_{34}$}
    \Geodesic[left]{A}{D}{$\delta_{14}$}
    \Geodesic[above,xshift=-3pt,yshift=2pt]{A}{C}{$\delta_{13}$}
    \Geodesic[above left,yshift=-4pt]{B}{D}{$\delta_{24}$}
  \end{scope}

  \draw[thick] (0,0) circle[radius=1];

  \IdealLabel[above left]{A}{$S_1$}
  \IdealLabel[above right]{B}{$S_2$}
  \IdealLabel[right]{C}{$S_3$}
  \IdealLabel[below left]{D}{$S_4$}

\end{tikzpicture}
\caption{Penner's theorem $\lambda_{13} \lambda_{24}=\lambda_{14} \lambda_{23} + \lambda_{12} \lambda_{34}$, where $\lambda_{ij}=e^{\delta_{ij}/2}$}
\label{figure1}
\end{figure}

Next we offer two results on Ptolemy's theorem for hyperbolic space, Theorems~\ref{theoremB1} and~\ref{theoremB2}, for collections of $n+1$ and $n+2$ points in $H^n$, respectively. The second theorem is known and was established by Valentine and Andalafte in \cite{VaAn1971}*{Theorem~4.7}; the first theorem could also be derived from their work. \rev{We offer short proofs of both results (in Section~\ref{ptolemy}) using Lorentzian space}. Valentine and Andalafte's work deserves to be better known because \rev{several} versions of Ptolemy's theorem that have followed since can be derived from their results.

\begin{maintheorem}{B1}[Ptolemy's theorem 1]\label{theoremB1}
Let $v_1,v_2,\dots,v_{n+1}$ be points in $H^n$ that belong to a common horosphere or hypersphere. Then $v_1,v_2,\dots,v_{n+1}$  lie on a common hyperplane if and only if $\det B=0$, where $B_{ij}=\sinh^2\tfrac12\varrho(v_i,v_j)$.
\end{maintheorem}

The case when $v_i$ belong to a common horosphere is equivalent to Euclidean versions of Ptolemy's theorem in several dimensions from \cite{Gr1996}*{Theorem~3} and \cite{MaMa2019}*{Theorem~3.4}. \rev{These results (paraphrased) say that, in $(n-1)$-dimensional Euclidean space $E^{n-1}$ (with distance function $d$), the points $v_1,v_2,\dots,v_{n+1}$ lie on a common Euclidean hypersphere or hyperplane if and only if $\det B'=0$, where $B_{ij}'=d(v_i,v_j)^2$. To see this equivalence, we recall that $E^{n-1}$ can be realised as a horosphere $\Sigma$ in $H^n$ with induced Euclidean metric  $d(u,v)=2\sinh\tfrac12\varrho(u,v)$, for points $u,v\in\Sigma$. Suppose that the points $v_i$ lie on $\Sigma$. Then $v_i$ belong to a common hyperplane in $H^n$ if and only if $v_i$ belong to a common Euclidean hypersphere or hyperplane in $E^{n-1}$. Furthermore, $\det B=0$ if and only if $\det B'=0$, so the equivalence is established.} Most of this geometric discussion can also be found at the end of \cite{VaAn1971}*{Section~3}.

Suppose now that $n=3$. Then $\det B=-4^{-4}\Delta (d_{14} d_{23} + d_{13} d_{24} + d_{12} d_{34})$, where $d_{ij}=d(v_i,v_j)$ and
\[
\Delta= (d_{14} d_{23} + d_{13} d_{24} - d_{12} d_{34})(d_{14} d_{23} - d_{13} d_{24} + d_{12} d_{34})(-d_{14} d_{23} + d_{13}d_{24} + d_{12} d_{34}).
\] 
We obtain Ptolemy's original theorem in the plane (and its converse), illustrated in Figure~\ref{figure2}, which says that four points $v_1$, $v_2$, $v_3$, and $v_4$ in a Euclidean plane lie on a common circle or line if and only if $d_{12} d_{34}=d_{14} d_{23} + d_{13} d_{24}$, $d_{13} d_{24}=d_{14} d_{23} + d_{12} d_{34}$, or $d_{14} d_{23}=d_{13}d_{24}+d_{12} d_{34}$.

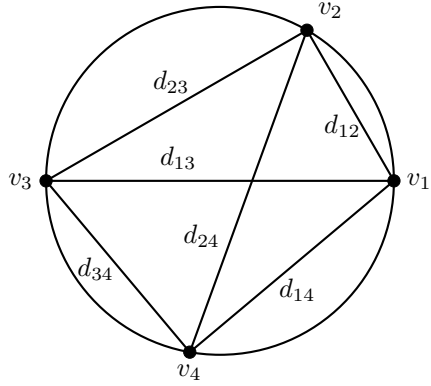
\begin{figure}[ht]
\centering
\begin{tikzpicture}[scale=2.3]
  \PtolemyVertex[right]{V1}{0}{$v_1$}
  \PtolemyVertex[above right]{V2}{60}{$v_2$}
  \PtolemyVertex[left]{V3}{180}{$v_3$}
  \PtolemyVertex[below]{V4}{260}{$v_4$}

  \DrawDiskBoundary

  \draw[ptolemy edge] (V1) -- (V2) -- (V3) -- (V4) -- cycle;
  \draw[ptolemy edge] (V1) -- (V3);
  \draw[ptolemy edge] (V2) -- (V4);

  \SegmentLabel[below=0pt, xshift=-3pt]{V1}{0.50}{V2}{$d_{12}$}
  \SegmentLabel[above=1pt, xshift=-2pt]{V2}{0.50}{V3}{$d_{23}$}
  \SegmentLabel[left=1pt, xshift=3pt,yshift=-1pt]{V3}{0.50}{V4}{$d_{34}$}
  \SegmentLabel[below=1pt, xshift=2pt]{V4}{0.50}{V1}{$d_{14}$}

  \SegmentLabel[above=1pt]{V1}{0.62}{V3}{$d_{13}$}
  \SegmentLabel[left=1pt]{V2}{0.64}{V4}{$d_{24}$}
\end{tikzpicture}
\caption{Ptolemy's theorem $d_{13} d_{24}=d_{14} d_{23} + d_{12} d_{34}$}
\label{figure2}
\end{figure}

The other case of Theorem~\ref{theoremB1}, when the points $v_i$ are assumed to lie on a (metric) hypersphere, gives us a spherical version of Ptolemy's theorem, which was obtained by Valentine \cite{Va1970b}*{Theorem~5.4} for the case $n=3$. More generally, the hypothesis in Theorem~\ref{theoremB1} that the $v_i$ belong to a common horosphere or hypersphere  can be broadened to allow the $v_i$ to belong to \rev{one branch of an equidistant hypersurface (with a minor modification to the present proof)}.

For our second version of Ptolemy's theorem in hyperbolic space, \rev{we use the notion of a \emph{branch of an equidistant hypersurface}, which is one of the two connected components of the locus of points in $H^n$ at a fixed positive distance from some hyperplane. We give an alternative, equivalent definition of such hypersurfaces in Section~\ref{ptolemy} using the Lorentzian inner product.}

\begin{maintheorem}{B2}[Ptolemy's theorem 2]\label{theoremB2}
Let $v_1,v_2,\dots,v_{n+2}$ be points in $H^n$. Then $v_1,v_2,\dots,v_{n+2}$  lie on a common horosphere, hypersphere,  hyperplane, or one branch of an equidistant hypersurface if and only if $\det B=0$, where $B_{ij}=\sinh^2\tfrac12\varrho(v_i,v_j)$.
\end{maintheorem}

\rev{We prove Theorem~\ref{theoremB2} by embedding $\mathbb{R}^{n,1}$ in $\mathbb{R}^{n+1,1}$ through a map $x\longmapsto x+e$, where $e$ is a spacelike unit vector orthogonal to $\mathbb{R}^{n,1}$. Under this embedding, the points $v_i$ in $H^n$ map to lightlike vectors in $\mathbb{R}^{n+1,1}$. The proof then resembles that of Theorem~\ref{theoremA}; indeed, the two theorems can be considered equivalent under this transformation.} 

Our fourth theorem is a (new) multidimensional version of Casey's theorem and its converse, stated in hyperbolic space. Each hyperplane in $H^n$ has two unit normal vector fields, one the negative of the other; \rev{employing terminology of differential geometry}, we define a \emph{cooriented hyperplane} in $H^n$ to be a hyperplane $\Pi$ with a choice of one of these two vector fields. For simplicity we denote a cooriented hyperplane by  the same notation $\Pi$ that we use for the underlying hyperplane itself. Now, let $\Pi$ and $\Pi'$ be two cooriented hyperplanes in $H^n$. When they are disjoint and not tangent at infinity, let $\gamma$ denote the common perpendicular to both hyperplanes, and when they intersect, let $\theta$ denote the angle between two normals at a point of intersection. We define
\[
\sigma(\Pi,\Pi') = 
\begin{cases}
\phantom{-}\sinh^2\tfrac12\varrho(\Pi,\Pi') & \text{if $\Pi\cap\Pi'=\varnothing$ and normal coorientations coincide along $\gamma$,}\\
-\cosh^2\tfrac12\varrho(\Pi,\Pi') & \text{if $\Pi\cap\Pi'=\varnothing$ and normal coorientations are opposite along $\gamma$,}\\
-\sin^2\tfrac12\theta & \text{if $\Pi\cap\Pi'\neq \varnothing$.}
\end{cases}
\]
When $\Pi$ and $\Pi'$ are tangent at $\infty$, the value of $\sigma(\Pi,\Pi')$ is 0 if the normal coorientations coincide along a common perpendicular horosphere and $-1$ if the normal coorientations are opposite along a common perpendicular horosphere. The somewhat opaque formula for $\sigma$ simplifies dramatically in Lorentzian space $\mathbb{R}^{n,1}$; this will be explained in Section~\ref{section4}, where we will see that $\sigma(\Pi,\Pi')$ is simply $\tfrac12(\langle v,v'\rangle-1)$, where $v$ and $v'$ are Lorentzian spacelike unit normals to $\Pi$ and $\Pi'$.

\begin{maintheorem}{C}[Casey's theorem]\label{theoremC}
Let $\Pi_1, \Pi_2, \dots, \Pi_{n+1}$ be hyperplanes in $H^n$. These hyperplanes can be endowed with coorientations such that $\det C=0$, where  $C_{ij} = \sigma(\Pi_i,\Pi_j)$, if and only if the hyperplanes 
\begin{enumerate}
    \item share a common tangent hyperplane at infinity,
    \item share a common ideal point, or
    \item are all orthogonal to one hyperplane and equally inclined to another.
\end{enumerate}
\end{maintheorem}

For clarity, statement (i) asserts that there is a hyperplane $\Pi$ whose ideal boundary is tangent (or equal) to the ideal boundary of each hyperplane $\Pi_i$. The assertion of statement (iii) that the hyperplanes are equally inclined means that all unoriented intersection angles are equal. \rev{The two hyperplanes in (iii) must be distinct.} These three statements have not (to our knowledge) previously appeared as a trio in any statement of Casey's theorem.  

To recover a version of Casey's theorem in $E^n$ from Theorem~\ref{theoremC}, we use the upper half-space model $\mathbb{H}^{n+1}=\{z+it:z\in\mathbb{R}^n,t>0\}$ of $H^{n+1}$ with Riemannian metric $|d\zeta|/t$, where $\zeta=z+it$. In this model the ideal boundaries of (hyperbolic) hyperplanes are (Euclidean) hyperspheres or extended hyperplanes in the one-point extension of $\mathbb{R}^n$. We define a \emph{cooriented hypersphere} (in $E^n$) to be a hypersphere $S$ with a sign $\varepsilon\in\{\pm 1\}$ (where $\varepsilon=1$ for the outward unit normal vector field and $\varepsilon=-1$ for the inward alternative). 

Consider cooriented hyperspheres $S_i$, for $i=1,2,\dots, n+2$, with radii $r_i$, centres $c_i$, and coorientations $\varepsilon_i$, and let $\Pi_i$ be the cooriented hyperplane with ideal boundary $S_i$ cooriented to agree with $S_i$. We define
\[
\tau(S_i,S_j) = \varepsilon_i\varepsilon_j(|c_i-c_j|^2-(r_i-\varepsilon_i\varepsilon_j r_j)^2).
\]
We recall that the signed inversive distance $\iota(S_i,S_j)=\varepsilon_i\varepsilon_j(|c_i-c_j|^2-r_i^2-r_j^2)/(2r_ir_j)$ between $S_i$ and $S_j$ equals $\pm\cosh \varrho (\Pi_i,\Pi_j)$ if the hyperplanes are disjoint (with $+$ for opposite coorientations and $-$ for coincident orientations) and $-\cos\phi$ if they intersect, where $\phi$ is the angle between the normals at a point of intersection. From this one can check that 
\[
\tau(S_i,S_j)=2r_ir_j(\iota(S_i,S_j)+1)=-4r_ir_j\sigma(\Pi_i,\Pi_j).
\]
Hence, with $D_{ij}=\tau(S_i,S_j)$ and $R=\operatorname{diag}(r_1,r_2,\dots,r_{n+2})$, we have $D=-4RCR$. Therefore $\det C=0$ if and only if $\det D=0$, so we have the following corollary of Theorem~\ref{theoremC}.

\begin{maincorollary}{D}[Casey's theorem in Euclidean space]\label{corollaryD}
Let $S_1,S_2,\dots,S_{n+2}$ be hyperspheres in $E^n$. These hyperspheres can be endowed with coorientations such that $\det D=0$, where $D_{ij} = \tau(S_i,S_j)$, if and only if the hyperspheres 
\begin{enumerate}
    \item share a common tangent hypersphere or hyperplane,
    \item share a common intersection point, or
    \item are all orthogonal to one hypersphere or hyperplane and equally inclined to another.
\end{enumerate}
\end{maincorollary}

\rev{Once more, in case (i) we allow the possibility that the common tangent hypersphere is one of the hyperspheres $S_i$ themselves.}

\rev{The principal novelty of Theorem~\ref{theoremC} and Corollary~\ref{corollaryD} is that they give a complete geometric characterisation of the vanishing determinant condition without additional assumptions. Even the complete characterisation in the planar case $n=2$ in Corollary~\ref{corollaryD} is apparently new. There are several antecedent results worth highlighting. In \cite{Za1942}, Zacharias establishes a planar version of Casey's theorem that includes the possibility of a common intersection point (as a limiting case of a common tangent circle) and has a converse under an additional hypothesis on triples. Kostin and Kostina \cite{KoKo2016}*{Theorem~4 and Sections~2--4} reinterpret Euclidean and hyperbolic Casey theorems as Ptolemy-type results about points on null quadrics -- this approach is close to our own. Abrosimov and Aseev \cite{AbAs2018}*{Theorem~4.1} obtain a version of Casey's theorem for four spheres in $E^n$ subject to separation assumptions. Then Maehara and Martini \cite{MaMa2019}*{Theorems~2.2--2.4}  derive versions of Casey's theorem with converses in $E^n$ for generic bipartite sets of spheres, subject to noncontainment and separation conditions. Most recently,  Kostin \cite{Ko2026} connects Penner's lambda lengths with hyperbolic Casey identities, anticipating  the unified treatment of Penner, Ptolemy, and Casey's theorems presented here.}

The three cases of Corollary~D are illustrated in Figure~\ref{figure3} for $n=2$. \rev{The first case shows four circles internally tangent to a common circle; in fact, Corollary~\ref{corollaryD} (and many versions of Casey's theorem) allow the circles to be internally and externally tangent to the common circle. }

\rev{To recover the original statement of Casey's theorem, suppose that, in case (i), the interiors of any pair of circles $S_i$ and $S_j$ are disjoint, and let}
\[
t_{ij}=\sqrt{|c_i-c_j|^2-(r_i-\varepsilon_i\varepsilon_jr_j)^2}.
\]
This is the length of the exterior tangent between $S_i$ and $S_j$ if $\varepsilon_i\varepsilon_j=1$ and the length of the interior tangent between $S_i$ and $S_j$ if $\varepsilon_i\varepsilon_j=-1$. Notice that $t_{ij}^2=\varepsilon_i\varepsilon_j\tau(S_i,S_j)$. Let $T$ be the four-by-four matrix with entries $t_{ij}^2$; then  $\det T=\det D$. When all $S_i$ are cooriented in the same way and have the cyclic labelling of Figure~\ref{figure3}(i), the equation $\det T=0$ implies that $t_{13}t_{24}= t_{12}t_{34}+t_{14} t_{23}$; we have thereby recovered the original theorem of Casey.

\begin{figure}[ht]
\begin{subfigure}[b]{0.3\textwidth}
\centering
\resizebox{\linewidth}{!}{%
\begin{tikzpicture}[scale=3]
  \IdealPointAngle{H1}{175}
  \IdealPointAngle{H2}{65}
  \IdealPointAngle{H3}{0}
  \IdealPointAngle{H4}{250}

  \SetHorocycleRadius{H1}{0.35}
  \SetHorocycleRadius{H2}{0.17}
  \SetHorocycleRadius{H3}{0.24}
  \SetHorocycleRadius{H4}{0.23}

  \begin{scope}
    \clip (0,0) circle[radius=1];

    \LabelledHorocycle{H1}{$S_1$}
    \LabelledHorocycle{H2}{$S_2$}
    \LabelledHorocycle{H3}{$S_3$}
    \LabelledHorocycle[yshift=-1pt]{H4}{$S_4$}

    \HorocycleTangent[below, xshift=1pt]{H1}{H2}{ 1}{$t_{12}$}
    \HorocycleTangent[above,xshift=5pt,yshift=-1pt]{H1}{H3}{ 1}{$t_{13}$}
    \HorocycleTangent[right, xshift=0pt]{H1}{H4}{-1}{$t_{14}$}
    \HorocycleTangent[left, xshift=1pt]{H2}{H3}{ 1}{$t_{23}$}
    \HorocycleTangent[right, xshift=-1pt]{H2}{H4}{-1}{$t_{24}$}
    \HorocycleTangent[above, xshift=-2pt]{H3}{H4}{ 1}{$t_{34}$}
  \end{scope}

  \DrawDiskBoundary
\end{tikzpicture}}
\caption{Common tangent circle}
\end{subfigure}
\hfill
\begin{subfigure}[b]{0.3\textwidth}
\centering
\resizebox{\linewidth}{!}{%
\begin{tikzpicture}[scale=2]
  \coordinate (P) at (0,0);

  \CircleThroughPoint{P}{K1}{180}{0.95}{$S_1$}
  \CircleThroughPoint{P}{K2}{ 95}{0.72}{$S_2$}
  \CircleThroughPoint{P}{K3}{  0}{0.54}{$S_3$}
  \CircleThroughPoint[yshift=-1pt]{P}{K4}{285}{0.36}{$S_4$}

  \filldraw[common point] (P) circle[radius=0.7pt];
\end{tikzpicture}}
\caption{Common intersection point}
\end{subfigure}
\hfill
\begin{subfigure}[b]{0.3\textwidth}
\centering
\resizebox{\linewidth}{!}{%
\begin{tikzpicture}[scale=2.6]
  \coordinate (O) at (0,0);
  \def\CaseThreeOrthogonalLineAngle{0}
  \def\CaseThreeEqualAngleLineAngle{58}
  \def\CaseThreeCircleRatio{1.28} 

  \NestedEqualAngleCircle{O}{N1}{\CaseThreeOrthogonalLineAngle}{0.88}{50}{$S_1$}
  \NestedEqualAngleCircle{O}{N2}{\CaseThreeOrthogonalLineAngle}{0.66}{50}{$S_2$}
  \NestedEqualAngleCircle{O}{N3}{\CaseThreeOrthogonalLineAngle}{0.47}{50}{$S_3$}
  \NestedEqualAngleCircle{O}{N4}{\CaseThreeOrthogonalLineAngle}{0.30}{50}{$S_4$}

  \LineThroughPointRange{O}{\CaseThreeOrthogonalLineAngle}{0.50}{2.18}
  \LineThroughPointRange{O}{\CaseThreeEqualAngleLineAngle}{0.65}{1.65}
\end{tikzpicture}}
\caption{Common angles}
\end{subfigure}  
\caption{Three cases of Casey's theorem}
\label{figure3}
\end{figure}
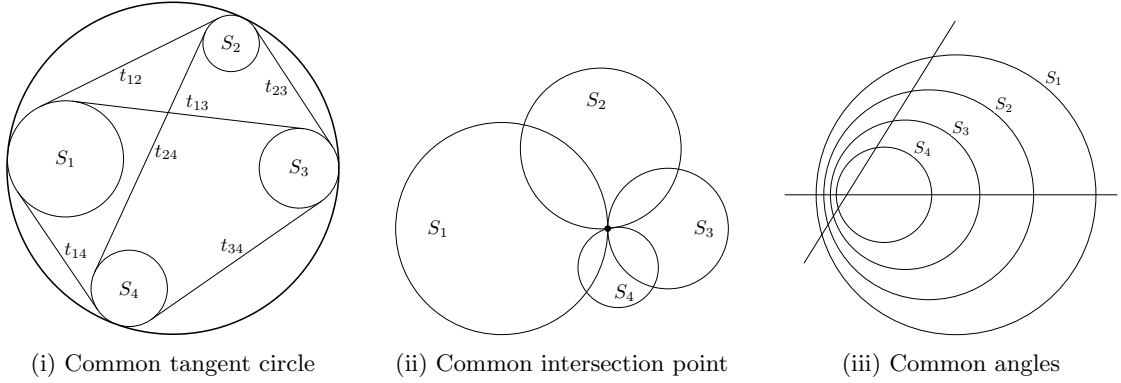

\rev{Theorem~\ref{theoremC} also gives spherical and hyperbolic versions of Casey's theorem for hyperspheres, each with a full converse. In brief, for the spherical version, identify the $n$-sphere with the ideal boundary of $H^{n+1}$, so that spherical hyperspheres correspond to ideal boundaries of hyperplanes. For the hyperbolic version, use inverse stereographic projection from the Poincar\'e ball model of $H^{n}$  to an open hemisphere of the $n$-sphere. The three geometric statements of Theorem~\ref{theoremC} transfer to three similar statements in each case; in the hyperbolic case, the auxiliary surfaces in cases (i) and (iii) may be hyperspheres, hyperplanes, horospheres, or branches of equidistant hypersurfaces. In dimension two, the hyperbolic and spherical Casey formulae of \cite{AbMi2015}*{Theorems~1 and~2} can be recovered as special cases of these results for four circles internally tangent to a common circle.}

\subsubsection*{Acknowledgements}

The first author was supported by an EPSRC DTP studentship and the second was supported by EPSRC grant EP/W002817/1. Both authors thank Katie Waddle for several constructive suggestions. We acknowledge the ICMS workshop \emph{Farey's legacy in frieze patterns and discrete geometry} (20--24 April 2026) during which some of this work was developed. 


\section{Penner's theorem}

We prove all our results using the Lorentzian model of hyperbolic space; see \cite{Pe2006} and \cite{Ra2019}*{Chapter 3} for background. Let $\mathbb{R}^{n,1}$ denote  $\mathbb{R}^{n+1}$ with the inner product
\[
\langle x,y\rangle = x_1y_1+x_2y_2+\dots+x_ny_n-x_{n+1}y_{n+1},
\]
and let $\|x\|^2=\langle x,x\rangle$. This is Lorentzian space of signature $(n,1)$. A nonzero vector $x$ is called \emph{spacelike} if $\|x\|^2>0$, \emph{timelike} if $\|x\|^2<0$, and \emph{lightlike} if $\|x\|^2=0$. We define $\mathcal{H}^n=\{x\in\mathbb{R}^{n,1}:\|x\|^2=-1,x_{n+1}>0\}$. This is a model of $n$-dimensional hyperbolic space with distance function $\varrho$, where $\cosh \varrho(x,y)= - \langle x,y\rangle$. The \emph{ideal boundary} of $H^n$ is the set of lines of lightlike vectors. 

A \emph{hyperplane} in $\mathcal{H}^n$ is the nonempty intersection of $\mathcal{H}^n$ with a codimension 1 subspace of $\mathbb{R}^{n,1}$.  A \emph{hypersphere} in $\mathcal{H}^n$ is a set $\{x\in \mathcal{H}^n: \varrho(x,c)=r\}$, for some $c\in \mathcal{H}^n$ and $r>0$. A \emph{horosphere} $\Sigma$ in $\mathcal{H}^n$ is a set $\Sigma=\{x\in\mathcal{H}^n: \langle x,v\rangle=-1/\sqrt2\}$, where $v$ is a lightlike vector in $\mathbb{R}^{n,1}$ with $v_{n+1}>0$. The line through $v$ is called the \emph{centre} of $\Sigma$.
 
The next elementary lemma, which involves a variation on a familiar calculation with Gram matrices, is central to proving all four main theorems. 

\begin{lemma}\label{lemma99}
The vectors $v_1,v_2,\dots,v_{n+1}\in \mathbb{R}^{n,1}$ belong to a codimension 1 subspace of $\mathbb{R}^{n,1}$ if and only if $\det X=0$, where $X_{ij}=\langle v_i,v_j\rangle$.
\end{lemma}
\begin{proof}
Let $Y$ be the $(n+1)$-by-$(n+1)$ matrix with rows given by the vectors $v_1,v_2,\dots,v_{n+1}$. These vectors belong to a codimension 1 subspace of $\mathbb{R}^{n,1}$ if and only if $\det Y=0$. Observe that $X=YDY^T$, where $D=\operatorname{diag}(1,1,\dots,1,-1)$. Hence $\det X=-(\det Y)^2$, and the result follows immediately. 
\end{proof}

We can now prove Theorem~\ref{theoremA} (using essentially the same reasoning from \cite{FeKaSeTu2023}).

\begin{proof}[Proof of Theorem~\ref{theoremA}]
The theorem clearly holds when all the centres coincide, so let us suppose this is not so. For each horosphere $\Sigma_i$, we choose a lightlike vector $v_i$ with $\Sigma_i=\{v\in\mathcal{H}^n: \langle v,v_i\rangle = -1/\sqrt{2}\}$. Then, following Penner \cite{Pe2006} and Felikson et al.\, \cite{FeKaSeTu2023}*{Remark~4.19}, we observe that  $\lambda(\Sigma_i,\Sigma_j)^2=-\langle v_i,v_j\rangle$. Lemma~\ref{lemma99} tells us that $\det A=0$, where $A_{ij}=\lambda(\Sigma_i,\Sigma_j)^2$, if and only if $v_1,v_2,\dots,v_{n+1}$ belong to a codimension 1 subspace of $\mathbb{R}^{n,1}$. Observe that the span of two linearly independent lightlike vectors contains a timelike vector. Consequently, if the vectors $v_i$ belong to a codimension 1 subspace of $\mathbb{R}^{n,1}$, then this subspace must intersect $\mathcal{H}^n$, so the centres through $v_i$ lie on the ideal boundary of a common hyperplane. Conversely, if the centres lie on the ideal boundary of a common hyperplane $W\cap \mathcal{H}^n$, for some codimension 1 subspace $W$, then each $v_i\in W$, so Lemma~\ref{lemma99} gives $\det A=0$.
\end{proof}

\section{Ptolemy's theorem}\label{ptolemy}

Here we prove  Theorems~\ref{theoremB1} and~\ref{theoremB2}. Recall that the hyperbolic distance $\varrho(u,v)$ between points $u,v\in\mathcal{H}^n$ satisfies $\cosh \varrho(u,v)= -\langle u,v\rangle$. Hence $\sinh^2\tfrac12\varrho(u,v)=-\tfrac12(\langle u,v\rangle+1)$.

\begin{proof}[Proof of Theorem~\ref{theoremB1}]
Suppose that $v_1,v_2,\dots,v_{n+1}$ lie on a common horosphere  in $H^n$. Then there is a lightlike vector $u$ with \rev{$\langle v_i,u\rangle=-1/\sqrt2$, for all $i$, and by rescaling $u$ we can ensure that $\langle v_i,u\rangle=-1/2$, for all $i$ (which suits our present purposes)}. The alternative is that $v_1,v_2,\dots,v_{n+1}$ lie on a common hypersphere in $H^n$. Then there is $v\in H^n$ and $s>1$ with $\langle v_i,v\rangle =-s$, for all $i$. Let $u=\delta v$, where $\delta = s+\sqrt{s^2-1}$. Then $\|u\|^2=-\delta^2\neq -1$ and
\[
\langle v_i,u\rangle =-\delta s = -\tfrac12(\delta^2+1)=\tfrac12(\|u\|^2-1).
\]
In both cases we have found $u\in\mathbb{R}^{n,1}$ with $\|u\|^2\neq \pm 1$ such that $\langle v_i,u\rangle=\tfrac12(\|u\|^2-1)$. Let $v_i'=v_i-u$. Then one can check that \(\langle v_i',v_j'\rangle = \langle v_i,v_j\rangle +1\). Hence
\begin{equation*}\label{eqn1}
B_{ij}=\sinh^2\tfrac12\varrho(v_i,v_j)=-\tfrac12\langle v_i',v_j'\rangle.
\end{equation*}
Now, if $v_1,v_2,\dots,v_{n+1}$ lie on a common hyperplane $\Pi$ in $H^n$, then there is a spacelike vector $w$ linearly independent from $u$ with $\langle v_i,w\rangle=0$, for all $i$. Let $w'=w-\lambda u$, where $\lambda =2\langle u,w\rangle/(\|u\|^2+1)$. Then one can check that $\langle v_i',w'\rangle=0$, for all $i$, and $w'\neq 0$, so $v_1',v_2',\dots,v_{n+1}'$ lie in a codimension 1 subspace of $\mathbb{R}^{n,1}$. Therefore $\det B=0$, by Lemma~\ref{lemma99}.

Suppose conversely that $\det B=0$. Then $v_1',v_2',\dots,v_{n+1}'$ lie in a codimension 1 subspace of $\mathbb{R}^{n,1}$, by Lemma~\ref{lemma99}, so there exists a nonzero vector $w'\in\mathbb{R}^{n,1}$ with $\langle v_i',w'\rangle=0$, for all $i$. Let $w=w'-\mu u$, where $\mu =2\langle u,w'\rangle/(\|u\|^2-1)$. Now, it cannot be that $w=0$, for if that were so then $w'=\mu u$ and $\langle v_i',w'\rangle=-\tfrac12\mu(\|u\|^2+1)$, which gives $\mu=0$ and hence $w'=0$. Therefore $w\neq 0$, and one can check that $\langle v_i,w\rangle =0$, for all $i$. Since $v_i$ are timelike it follows that $w$ is spacelike, so $v_i\in\Pi$, where $\Pi$ is the hyperplane $\{x\in\mathcal{H}^n:\langle x,w\rangle =0\}$, as required.
\end{proof}

Next we prove Theorem~\ref{theoremB2}. In the Lorentzian model of hyperbolic space, a branch of an equidistant hypersurface is a set  $\{x\in\mathcal{H}^n:\langle x, v\rangle = \lambda\}$, where $v$ is spacelike and $\lambda\neq 0$.

\begin{proof}[Proof of Theorem~\ref{theoremB2}]
We begin by embedding $\mathbb{R}^{n,1}$ in $\mathbb{R}^{n+1,1}$ by the map $(x_1,x_2,\dots,x_{n+1})\longmapsto (0,x_1,x_2,\dots,x_{n+1})$. Let $u=(1,0,\dots,0)$ and $v_i'=v_i+u$. Then  \(\langle v_i',v_j'\rangle = \langle v_i,v_j\rangle +1\), so, as before, $B_{ij}=-\tfrac12\langle v_i',v_j'\rangle$. By Lemma~\ref{lemma99}, we have $\det B=0$ if and only if the lightlike vectors $v_i'$ lie in a codimension 1 subspace of $\mathbb{R}^{n+1,1}$. 

Suppose that $v_i$ lie on a common horosphere, hypersphere,  hyperplane, or one branch of an equidistant hypersurface. Then there is a nonzero vector $v\in\mathbb{R}^{n,1}$ and $\mu\in\mathbb{R}$ with $\langle v_i,v\rangle=\mu$, for all $i$. Let $w=-\mu u+v$; then $\langle v_i',w\rangle=\langle v_i,v\rangle -\mu\langle u,u\rangle=0$, so $v_i'$ lie on a codimension 1 subspace of $\mathbb{R}^{n+1,1}$. 

Conversely, suppose that $v_i'$ belong to a codimension 1 subspace of $\mathbb{R}^{n+1,1}$. Since two linearly independent lightlike vectors span a plane containing a timelike vector, the codimension 1 subspace containing the $v_i'$ has a spacelike normal $w$, and such a normal also exists if the $v_i'$ all lie on a lightlike line. Let us write $w=\lambda u+v$, where $\lambda\in\mathbb{R}$, $v\neq 0$, and $\langle u,v\rangle=0$. Since $w$ is spacelike, $\lambda^2+\|v\|^2>0$. If $v$ is lightlike, then the equation $\langle v_i,v\rangle = -\lambda$ shows that all $v_i$ lie on a horosphere ($\lambda\neq 0$ in this case). If $v$ is spacelike, then the $v_i$ lie on a hyperplane if $\lambda=0$ and a branch of an equidistant hypersurface if $\lambda\neq 0$. If $v$ is timelike, then from $\lambda^2>-\|v\|^2$ we see that $v_i$ lie on a hypersphere.
\end{proof}

\section{Casey's theorem}\label{section4}

In this section we prove Theorem~\ref{theoremC}. Let $\Pi$ and $\Pi'$ be cooriented hyperplanes with spacelike unit normals $v$ and $v'$. When $\Pi$ and $\Pi'$ are disjoint and not tangent at infinity, we have $|\langle v,v'\rangle|=\cosh\varrho(\Pi,\Pi')$. Let $\gamma$ be the geodesic that intersects $\Pi$ and $\Pi'$ orthogonally. Then $\langle v,v'\rangle$ is $\cosh\varrho(\Pi,\Pi')$ if the normal coorientations coincide along $\gamma$ and $-\cosh\varrho(\Pi,\Pi')$ otherwise. Hence $\langle v,v'\rangle-1=2\sinh^2\tfrac12\varrho(\Pi,\Pi')$ in the first case and $\langle v,v'\rangle-1=-2\cosh^2\tfrac12\varrho(\Pi,\Pi')$ in the second. On the other hand, when $\Pi$ and $\Pi'$ intersect, we have $\langle v,v'\rangle=\cos\theta$, where $\theta$ is the angle between $v$ and $v'$, so $\langle v,v'\rangle-1=-2\sin^2\tfrac12\theta$. It follows that $
\langle v,v'\rangle-1 = 2\sigma(\Pi,\Pi')$. From this we see that if the hyperplanes $\Pi_i$ are cooriented with spacelike unit normals $v_i$, then 
\[
\langle v_i,v_j\rangle -1= 2\sigma(\Pi_i,\Pi_j)= 2C_{ij}.
\]
Let us also frame statements (i) to (iii) of Theorem~\ref{theoremC} in Lorentzian terms. For a spacelike vector $v$, we define $v^\perp=\{w\in\mathbb{R}^{n,1}:\langle w,v\rangle =0\}$; this is a codimension 1 subspace of $\mathbb{R}^{n,1}$ which intersects $\mathcal{H}^n$ in a hyperplane. Statement (i), that the hyperplanes $\Pi_i$ share a common tangent hyperplane at infinity, is equivalent to the statement that there is a spacelike unit normal $v$ with $|\langle v_i,v\rangle|=1$, for all $i$. Statement (ii), that the hyperplanes $\Pi_i$ share a common ideal boundary point, is equivalent to the statement that there is a lightlike vector $w$ with $\langle v_i,w\rangle=0$, for all $i$. Finally, statement (iii), that the hyperplanes $\Pi_i$ are all orthogonal to one hyperplane  and equally inclined to another, is equivalent to the statement that there are linearly independent spacelike unit normals $u$ and $v$  and $0\leq \lambda<1$ with $\langle v_i,u\rangle=0$ and $|\langle v_i,v\rangle|=\lambda$, for all $i$. These three statements concern the \emph{unoriented} hyperplanes $\Pi_i$, and accordingly the reformulations do not depend on the choice of sign for the spacelike unit normals~$v_i$.

\begin{proof}[Proof of Theorem~\ref{theoremC}]
We begin by assuming that one of the statements (i), (ii), or (iii) holds. Let $v_i$ be a spacelike unit normal to $\Pi_i$, for $i=1,2,\dots,n+1$; in the arguments that follow we will modify this initial, arbitrary choice of unit normals to fit our purpose. 

Suppose first that statement (i) holds: the hyperplanes $\Pi_i$ share a common tangent hyperplane at infinity.  Then there exists a unit spacelike vector $v$ with $\langle v_i,v\rangle =\pm 1$, for  all $i$. By changing the sign of some vectors $v_i$ we can ensure that $\langle v_i,v\rangle =1$, for all $i$. Let $w_i=v_i-v$. Then 
\(
\langle w_i,w_j\rangle = \langle v_i,v_j\rangle - 1=2C_{ij}
\)
and $\langle w_i,v\rangle=0$. Hence $w_i$ lies in the codimension 1 subspace $v^\perp$ of $\mathbb{R}^{n,1}$, so $\det C=0$, by Lemma~\ref{lemma99}.

Suppose now that statement (ii) holds: the hyperplanes $\Pi_i$ share a common ideal point. Then there exists a lightlike vector $v$ with $\langle v_i, v \rangle = 0$, for all $i$. After rescaling we can write $v=v_S+v_T$, where $v_T=(0,0,\dots,0,1)$ and $\langle v_S,v_T\rangle=0$. Let $w_i=v_i+\langle v_i,v_T\rangle v-v_T$. Then one can check that $\langle w_i,w_j\rangle=\langle v_i,v_j\rangle-1$ and $\langle w_i,v_S\rangle=0$. Hence $w_i\in v_S^\perp$, so $\det C=0$, by Lemma~\ref{lemma99}.

Last, suppose that statement (iii) holds, in which case the hyperplanes $\Pi_i$ are all orthogonal to one hyperplane and equally inclined to another. Then there is a real number $\lambda$ with $|\lambda|<1$ and linearly independent unit spacelike vectors $u$ and $v$ such that $\langle v_i,u\rangle=0$ and $\langle v_i,v\rangle =\lambda$, for all $i$ (after adjusting signs of the normals $v_i$). We can write $v=su+v^*$, where $s\in\mathbb{R}$ and $\langle u,v^*\rangle=0$. This case now splits in two depending on whether or not $v^*$ is lightlike. Suppose first that it is lightlike. If $\lambda=0$, then the hyperplanes $\Pi_i$ share the common ideal point $v^*$, so case (ii) holds. If $\lambda\neq 0$, then we can define $w_i=v_i-v^*/(2\lambda)$. Then $w_i\in u^\perp$ and $\langle w_i,w_j\rangle=\langle v_i,v_j\rangle -1=2C_{ij}$, so $\det C=0$, by Lemma~\ref{lemma99}.

Suppose now that $v^*$ is not lightlike. Let $w_i=v_i-(\lambda/\|v^*\|^2)v^*$. Then $w_i$ is perpendicular to both $u$ and $v$. Furthermore, $\langle w_i,w_j\rangle = \langle v_i,v_j\rangle-\lambda^2/\|v^*\|^2$, so $C_{ij}=\tfrac12(X_{ij}+Y_{ij})$, where $X_{ij}=\langle w_i,w_j\rangle$ and $Y_{ij}=\lambda^2/\|v^*\|^2-1$. Observe that $\operatorname{rank}(X)\leq n-1$ because $w_i\in u^\perp\cap v^\perp$ and $\operatorname{rank}(Y)\leq 1$. Hence $\operatorname{rank}(C)\leq n$, so $\det C=0$.

For the converse, suppose there are spacelike unit normals $v_i$ to the hyperplanes $\Pi_i$ for which $\det C=0$, where $C_{ij}=\tfrac12(\langle v_i,v_j\rangle -1)$. Then there exists $(\lambda_1,\lambda_2,\dots,\lambda_{n+1})\in\mathbb{R}^{n+1}\setminus\{0\}$ with
\[
\sum_{j=1}^{n+1}\lambda_j\langle v_i,v_j\rangle = \lambda, \quad \text{for $i=1,2,\dots,n+1$,}
\]
where $\lambda=\lambda_1+\lambda_2+\dots+\lambda_{n+1}$. Let $v=\sum_{j=1}^{n+1}\lambda_jv_j$. By scaling when $v\neq 0$ we can assume that $\|v\|^2$ either is 0 or has modulus 1. Observe that $\langle v_i,v\rangle=\lambda$, for all $i$, and
\[
\|v\|^2=\Big\langle \sum_{i=1}^{n+1}\lambda_iv_i,v\Big\rangle = \sum_{i=1}^{n+1}\lambda_i\langle v_i,v\rangle =\sum_{i=1}^{n+1}\lambda_i\lambda=\lambda^2.
\]
Hence either $\|v\|^2=1$ and $\lambda=\pm 1$ or $\|v\|^2=0$ and $\lambda=0$. In the first case we have $\langle v_i,v\rangle =1$ for all $i$ or $\langle v_i,v\rangle =-1$ for all $i$. Then all the hyperplanes $\Pi_i$ are tangent to the hyperplane $v^{\perp}$ at infinity (statement (i)). In the second case we have $\|v\|^2=0$ and $\lambda=0$. If $v\neq 0$ then $\langle v_i,v\rangle=0$, for all $i$, so $v$ is an ideal point of each of the hyperplanes $\Pi_i$ (statement (ii)).

The remaining possibility is that $v=0$. In this case $\sum_{j=1}^{n+1} \lambda_jv_j = 0$ and $\sum_{j=1}^{n+1}\lambda_j=0$, so the dimension of the subspace $V$ spanned by $v_i-v_{n+1}$, for $i=1,2,\dots,n$, is at most $n-1$. Consequently, $\operatorname{dim} V^\perp\geq 2$. Consider the linear functional $\phi(x)=\langle v_{n+1},x\rangle$ on $V^\perp$. Since $\operatorname{dim} V^\perp\geq 2$, there exists a nonzero vector $u\in\operatorname{ker} \phi$. Hence
\[
\langle v_i,u\rangle = \langle v_i-v_{n+1},u\rangle+\langle v_{n+1},u\rangle=0, 
\]
for all $i$. Now, since $\operatorname{dim} V^\perp\geq 2$ we can choose a vector $w\in V^\perp$ linearly independent from $u$. Let $\alpha=\langle v_{n+1},w\rangle$. Then $\langle v_i,w\rangle=\alpha$, for all $i$.

To complete the proof we consider the cases when $u$ is lightlike, spacelike, and timelike, in turn. Suppose first that $u$ is lightlike. Then $u$ is an ideal point common to each hyperplane $\Pi_i$ (statement (ii)). Suppose instead that $u$ is spacelike; by scaling we can assume that $\|u\|^2=1$. Let $z=w+tu$, for $t\in\mathbb{R}$. Then $\|z\|^2=\|w\|^2+2t\langle w,u\rangle +t^2$, so we can choose $t$ with $\|z\|^2>\alpha^2$. Let $w^*=z/\sqrt{\|z\|^2}$ and $\alpha^*=\alpha/\sqrt{\|z\|^2}$, so  $|\alpha^*|<1$. Then $w^*$ is a unit spacelike vector and $\langle v_i,w^*\rangle = \alpha^*$, for all $i$. Therefore each hyperplane $\Pi_i$ intersects $u^\perp$ orthogonally and intersects $(w^*)^\perp$ in a common unoriented angle $\theta$, where $\lvert\cos\theta\rvert=|\alpha^*|$ (statement (iii)).

Suppose finally that $u$ is timelike; by scaling we can assume that $\|u\|^2=-1$. Since $u$ is timelike the subspace $U=u^\perp$ is spacelike. We have $v_i\in U$, for all $i$. Let $y=w+\langle w,u\rangle u$. Then $\langle y,u\rangle=0$, so $y\in U$, and $y\neq 0$ since $u$ and $w$ are linearly independent. Observe that $\langle v_i,y\rangle =\alpha$, for all $i$. By scaling the spacelike vector $y$ (and scaling $\alpha$ accordingly) we can assume that $\|y\|^2=1$. After this scaling we have that $|\alpha|=|\langle v_i,y\rangle|\leq 1$. If $\alpha=\pm 1$, then $v_i=\pm y$, for each $i$, so all the hyperplanes $\Pi_i$ coincide and statement (i) is satisfied. If $\alpha=0$, then we can apply the argument for $u$ spacelike from the preceding paragraph (now with $y$ in place of $u$ and $u$ in place of $w$) to give statement (iii). 

Let us assume, then, that $\alpha\neq -1,0,1$. Let $z=\alpha^{-1}y+\sqrt{\alpha^{-2}-1}u$. Then $\|z\|^2=1$, so $z$ is spacelike, and $\langle v_i,z\rangle=1$, for all $i$. Hence all hyperplanes $\Pi_i$ are tangent at infinity to the hyperplane $z^\perp$ (statement (i)).
\end{proof}

\end{document}